\documentclass[12pt,leqno]{amsart}
\usepackage{latexsym,amssymb,amsmath,multicol, rotating,lscape}
\usepackage{letltxmacro}

\usepackage{tikz}
\usetikzlibrary{calc,decorations.markings}

 2

\newtheorem{thm}{Theorem}[section]
\newtheorem{lem}[thm]{Lemma}

\newcommand{\thmref}[1]{Theorem~\ref{#1}}
\newcommand{\lemref}[1]{Lemma~\ref{#1}}

\theoremstyle{remark}
\newtheorem{rmk}{Remark}[section]

\begin{document}

\title[Averages and sign change  of Fourier coefficients of cusp forms]
{Average estimates and sign change of  Fourier coefficients of cusp forms at integers represented by binary quadratic form of fixed discriminant }

\author{Lalit Vaishya}
\address{Harish-Chandra Research Institute,  HBNI,  
         Chhatnag Road, Jhunsi,
         Prayagraj - 211 019,
         India \\ \newline
        \textbf{Current address:} School of Mathematical Sciences, 
National Institute of Science Education and Research, Bhubaneswar, HBNI,  
         Bhubaneswar- 752050
         India }

\email[Lalit Vaishya]{ lalitvaishya@gmail.com,  lalitvaishya@niser.ac.in}


\date{\today}
 
\maketitle

\begin{abstract}

In this article, we establish an average behaviour of the normalised Fourier coefficients of the Hecke eigenforms supported at the integers represented by any primitive integral positive definite binary quadratic form of fixed discriminant $D < 0$ when the class number $h(D) = 1$. 

 We also obtain a quantitative result for the number of sign changes of the sequence of the normalised Fourier coefficients $\lambda_{f}(n)$ of the Hecke eigenforms $f$ where $n$ is represented by  any primitive integral positive definite binary quadratic form of fixed discriminant $D < 0$ when the class number $h(D) = 1$ in the interval $(x,2x]$, for sufficiently large $x$.
 

\end{abstract}

\section{Introduction and statements of the results}

Let $  SL_2(\mathbb{Z}):=\left\{\left(\begin{array}{cc}
a & b\\
c & d
\end{array}\right): a, b, c, d \in {\mathbb Z}, \quad ad-bc=1\right\}$ be the full modular group,  \\ $\Gamma_{0}(N) :=\left\{\left(\begin{array}{cc}
a & b\\
c & d
\end{array}\right) \in  SL_2(\mathbb{Z}):   c \equiv 0 \pmod N\right\}$ be a congruence subgroup of level $N$ and $\mathbb{H} :=\left\{z\in\mathbb{C}: \Im (z)>0\right\}$ denote the complex upper half plane. 
Let $M_{k}(\Gamma_{0}(N), \chi)$ ($S_{k}(\Gamma_{0}(N), \chi)$) denote the ${\mathbb C}$- vector space of modular forms (cusp forms) of weight $k$ on the congruence subgroup $\Gamma_{0}(N)$ with the Dirichlet character $\chi$.  A cusp form $f \in S_{k}(\Gamma_{0}(N), \chi)$ is said to be a primitive cusp form (Hecke eigenform) if $f$ is a common eigenform for all the Hecke operators and the Atkin - Lehner $W$-operators.
 Let $S_{k}(SL_2(\mathbb{Z}))$ denote the ${\mathbb C}$- vector space of cusp forms of weight $k$ for the full modular group. 
   Let $f \in S_{k}(SL_2(\mathbb{Z}))$ be a normalised Hecke eigenform with the Fourier series expansion at the cusp $\infty$;
\begin{equation}\label{Hol-Four Exp}
f(\tau)=\sum_{n=1}^\infty \lambda_f(n)n^{\frac{k-1}{2}}q^{n},
\end{equation}
where $q:=e^{2\pi i \tau}$. $f$ is said to normalised if $\lambda_f(1)=1$. The Fourier coefficient $\lambda_f(n)$ is called the normalised $n^{th}$-Fourier coefficient of $f$. It is well-known that the Fourier coefficient $\lambda_{f}(n)$ is a multiplicative function and satisfies the following recursive relation \cite[Eq. (6.83)]{Iwaniec}:
\begin{equation}
\lambda_{f}(m)\lambda_{f}(n) = \sum_{d \vert \gcd(m,n)}  \lambda_{f}\left(\frac{mn}{d^2}\right),  \qquad  {\rm \ for \  all  \ positive \ integers }\ m \ {\rm  and } \ n.
\end{equation}
  It also satisfies  Deligne bound, i.e., for  any $\epsilon > 0$, we have
\begin{equation} \label{lambda-coefficient-bound}
|\lambda_{f}(n)| \le d(n) \ll_{\epsilon} n^{\epsilon}, 
\end{equation}
where $d(n)$ denotes the number of positive divisors of $n$.

\smallskip

The Fourier coefficients of cusp forms are fascinating objects in the world of automorphic forms and contain a lot of interesting informations. In analytic number theory, it is a well known approach to look at summatory function of arithmetical  functions over certain sequences.  R. A. Rankin \cite{Rankin} and A. Selberg \cite{Selberg}, independently developed a method (now referred to as the Rankin-Selberg method) and obtained the following estimates: 
\begin{equation*}
\begin{split}
\sum_{n \le x} (\lambda_{f}(n))^{2} = c x + O_{f, \epsilon}(x^{\frac{3}{5}+\epsilon}).
\end{split}
\end{equation*}
This is the best known estimate. C. J.  Moreno and F. Shahidi \cite{Moreno-Shahidi} has obtained the following estimates using the properties of symmetric power $L$-functions: \quad for sufficiently large $x,$
\begin{equation*}
\begin{split}
\sum_{n \le x} (\tau_{0}(n))^{4} \ll x \log x, 
\end{split}
\end{equation*}
where $\tau_{0}(n) = \tau(n) n^{-\frac{11}{2}}$ is the normalised Ramanujan tau function. The result of Moreno and Shahidi also hold true for the Fourier coefficients $\lambda_{f}(n)$'s of a normalised Hecke eigenform $f$ of any integral weight.  O. M. Fomenko \cite{Fomenko} has established estimates  for  $\displaystyle{\sum_{n \le x}(\lambda_{f}(n))^{r}},$  when $r= 2, 4$ and improved the result of  Moreno and Shahidi \cite{Moreno-Shahidi}.  G. L\"{u} [\cite{G. Lu}, \cite{G. Lu1}, \cite{G. Lu2}] improved Fomenko's results and generalised the results for higher moments, i.e., for $r \ge 2$. Y. K. Lau et al. \cite{Lau-Lu-Wu} considered more general cases and obtained better estimates. Recently, S. Zhai \cite{Zhai} has obtained an estimate for the  power sum  of the Fourier coefficients  $\lambda_{f}(n)$ over the sum of two squares, i.e., estimates of the following sums: for $ 2 \le r \le 8$ and $x \ge 1$,
 $$
 \displaystyle{\sum_{(a, b) \in {\mathbb Z}^{2} \atop a^{2}+b^{2} \le x} (\lambda_{f}(a^{2}+b^{2}))^{r}}.
$$ 

In this article, we generalise the work of S. Zhai \cite{Zhai} to the case where the sum of two squares is replaced by any integer represented by a primitive integral positive definite binary quadratic form of fixed negative discriminant $D$ with the class number $h(D) =1$. We also give an estimate when $r =1.$ Precisely, we prove an estimate for the following sums:

\smallskip
\noindent for  $1\le r \le 8 \quad {\rm  and} \quad  x \ge 1$, 
\begin{equation}
\begin{split}
\displaystyle{\sum_{\underline{x} \in {\mathbb Z}^{2} \atop Q(\underline{x}) \le x} (\lambda_{f}(Q(\underline{x})))^{r}}, \\
\end{split}
\end{equation}
 where $\lambda_{f}(n)$ is the $n^{th}$  normalised Fourier coefficients of a Hecke eigenform $f$ and  $Q(\underline{x})$ is a primitive integral positive definite binary quadratic form (reduced form) of fixed negative discriminant $D$ with the class number $h(D) =1.$ 

\smallskip
Moreover,  we use these estimates to obtain a result on sign change of sequence of  the Fourier coefficients of the Hecke eigenforms supported at integers represented by a primitive integral binary quadratic form of fixed discriminant $D<0$ with class number $h(D) = 1$. In this work, we also improve the result obtained in  [\cite{manish-soumya}, \cite{Lalit}]. We refer  to the introduction  in \cite{manish-soumya}, \cite{Lalit}  for the history on the problem of sign change of the Fourier coefficients of cusp form.  Now, we fix some notations and state our results. 

\bigskip
Let $B(x_{1}, x_{2})$ be an integral positive definite binary quadratic form given by $B(x_{1}, x_{2}) := ax_1^2 + b x_1x_2 + c x_2^2,$ where $ (x_1,x_2)\in \mathbb{Z}^2$, $a,b,c \in \mathbb{Z}$ with fixed discriminant $D = b^2-4ac < 0$. Two forms $B_{1}(x_{1}, x_{2})$ and $B_{2}(x_{1}, x_{2})$ are said to be equivalent if there are integers $p, q, r$ and $s$ with $p s - q r = \pm 1$ such that $B_{1}(x_{1}, x_{2}) = B_{2}(p x_{1} + q x_{2},  r x_{1} + s x_{2})$. An integral binary quadratic form $B(x_{1}, x_{2})$ is said to be a primitive form if $\gcd(a, b, c) = 1$.  A primitive integral binary quadratic form $B(x_{1}, x_{2})$ is said to be a reduced form if $|b| \le a \le c$.  For a fixed $D < 0$, let $H(D)$ denote the set of equivalence classes of primitive integral binary quadratic form of discriminant $D$. $H(D)$ forms a finite abelian group under some composition law. Let $h(D)$ denote the class number for the discriminant $D,$ i.e.  $h(D) := \# H(D).$  For a given discriminant $D < 0$, the class number $h(D)$ is finite and it is given by the number of reduced forms of discriminant $D$ \cite[Theorem 2.13]{Cox}.
  For details, we refer to \cite[chapter 2]{Cox}. 
 
\smallskip
Throughout this paper, $Q(\underline{x})$ denotes a primitive integral positive definite binary quadratic (reduced) form given by $Q(\underline{x}) = ax_1^2 + b x_1x_2 + c x_2^2,$ where $\underline{x} = (x_1,x_2)\in \mathbb{Z}^2$, $a, b, c \in \mathbb{Z}$ with $\gcd(a, b, c) = 1$ and  fixed discriminant $D = b^2-4ac < 0$. Further,  we assume that for such discriminant $D$, the class number  $h(D)$ is $ 1$. 

\begin{rmk} 
 For the primitive integral positive definite  binary quadratic form of the following fundamental discriminants $D = -3, -4, -8, -7, -11, -19, -43, -67, -163$ and non-fundamental discriminants $D =  -12, -16, -27, -28,$ the class number $h(D) = 1$. The list of reduced form corresponding to these discriminants are given in  \cite[Page 19- 20]{Buell}.

\end{rmk}

We define the following generating function  $\theta_{Q}(\tau)$ associated to binary quadratic  form $Q(\underline{x})$ given by: 
\begin{equation*}
\begin{split}
\theta_{Q}(\tau) & := \displaystyle{\sum_{\underline{x} \in {\mathbb Z}^{2} } q^{(Q(\underline{x}))}}  \quad = \quad  \displaystyle{\sum_{n = 0}^{\infty} r_{Q}(n) q^{n}}, \qquad   \quad ~~~ q = e^{2 \pi i \tau}, \\
\end{split}
\end{equation*} 
where $r_{Q}(n)$ denotes the number of representations of a positive integer $n$ by the quadratic form $Q(\underline{x})$, i.e.,  $r_{Q}(n) = \# \{\underline{x} \in  {\mathbb Z}^{2} : n= Q(\underline{x})  \}$. The generating function $\theta_{Q}(\tau) \in M_{1}(\Gamma_{0}(|D|), \chi_{D}) $ (see \cite[Theorem 10.9]{Iwaniec}) where $\chi_{D}$ is the Dirichlet character modulo $|D|$ which  is given by Jacobi symbol $\chi_{D}(d) := \left(\frac{D}{d} \right)$.  From  Weil's bound,  we have  $r_{Q}(n) \ll n^{\epsilon},$ for any arbitrarily small  $\epsilon > 0.$

\smallskip
We define the character sum  $r(n; D)$ (given explicitly) in terms of Jacobi symbol $\chi_{D}$ (see \cite[Eqs.(11.9), (11.10)]{Iwaniec}), i.e.,
\begin{equation} \label{r-quad}
r(n; D) = w_{D} \sum_{d \vert n} \chi_{D}(d), \   \
{\rm ~where~} \ \ \ \ \ 
w_{D} =
\begin{cases} 
6 {\rm ~~~~~ if~~~~~~} D = -3,\\
4 {\rm ~~~~~ if~~~~~~} D = -4,\\
2 {\rm ~~~~~ if~~~~~~} D < -4. \end{cases} 
\end{equation}

\begin{rmk}
The formula for $r(n; D)$ depends only on the  discriminant $D$ ($= b^2-4ac < 0$) and not on the choice of $a, b$ and $c$.
\end{rmk}

In this work, we are considering the quadratic form $Q(\underline{x})$ of discriminant $D < 0$ such that $h(D) = 1.$ So the formula $r(n; D)$ given in \eqref{r-quad} agree with the formula for number of representations $r_{Q}(n)$  of n by the quadratic form $Q(\underline{x})$ \cite[section 11.2]{Iwaniec}. Thus, we have
\begin{equation} \label{rr-quad}
r_{Q}(n) = r(n; D) = w_{D} \sum_{d \vert n} \chi_{D}(d).\\
\end{equation}

\smallskip

Let $f (\tau) = \displaystyle{\sum_{n=1}^\infty \lambda_f(n)n^{\frac{k-1}{2}}q^{n}} \in  S_{k}(SL_2(\mathbb{Z}))$ 
be a normalised Hecke eigenform with the Fourier series expansion at cusp $\infty$ and $Q(\underline{x}) $  be a primitive integral positive definite binary quadratic from (reduced form) with fixed discriminant $D < 0$ and class number $h(D) =1$. 

\bigskip
 \noindent
\qquad For $1 \le r \le 8$ and $x \ge 1$, let 
\begin{equation}\label{Sum-PIBQF}
\begin{split}
S_{r}(f, D; x ) := \displaystyle{\sum_{\underline{x} \in {\mathbb Z}^{2} \atop Q(\underline{x}) \le x} (\lambda_{f}(Q(\underline{x})))^{r}}. \\
\end{split}
\end{equation}
\begin{thm}\label{Estimate-BFQ}
Let $\epsilon >0$ be an arbitrarily small real number.  For $1 \le r \le 8$  and for sufficiently large $x$,  we have
\begin{equation} \label{SR-Est}
\begin{split}
S_{r}(f, D; x ) = x P_{r}(\log x) + O_{f,D,\epsilon}(x^{\gamma_{r}+\epsilon}),\\
\end{split}
\end{equation}
where $P_{2}(x), P_{4}(x), P_{6}(x)$ and $P_{8}(x)$ are polynomials of degree $0, 1, 4$ and $13$ respectively, $P_{r}(t) \equiv 0$ for $r =1, 3, 5, 7$ and 
\begin{center}
$\gamma_{1} = \frac{7}{10}$, \qquad $\gamma_{2} = \frac{8}{11}$, \qquad $\gamma_{3} = \frac{17}{20}$, \qquad \qquad  \qquad  $\gamma_{4} = \frac{43}{46}$, \\  $\gamma_{5} = \frac{83}{86}$, \qquad $\gamma_{6} = \frac{184}{187}$, \qquad $\gamma_{7} = \frac{355}{358}$ \qquad and \qquad $\gamma_{8} = \frac{752}{755}$. 
\end{center}
 \end{thm}

As a consequence of this theorem, we obtain the following result on sign change of  the Fourier coefficients of $f \in S_{k}(SL_2(\mathbb{Z}))$ supported at integers represented by a primitive integral positive definite binary quadratic form of fixed negative discriminant with the class number $1$. This improves our earlier result.

\begin{thm}\label{sign change-BFQ}
Let $k \ge 2$ be an integer and $f(\tau) = \displaystyle{\sum_{n=1}^{\infty} \lambda_{f}(n) n^{\frac{k-1}{2}} q^n} \in S_{k}(SL_2(\mathbb{Z}))$ be a normalised Hecke eigenform. Then, the sequence $\{\lambda_{f}(Q(\underline{x}))\}_{\underline{x} \in {\mathbb{Z}^2} }$ has infinitely many sign changes where $Q(\underline{x})$ is a primitive integral positive definite binary quadratic form with fixed discriminant $D < 0$ and class number $h(D) = 1$. \\
Moreover, for any arbitrarily small constant $\epsilon >0$, there are at least $x^{\frac{8}{33}-\epsilon}$ many sign changes of the sequence $\{\lambda_{f}(Q(\underline{x}))\}_{\underline{x} \in {\mathbb{Z}^2} }$ such that $Q(\underline{x})$ lies  in the interval $(x, 2x]$, for sufficiently large $x$. 
\end{thm}

\begin{rmk}
In \cite{Lalit}, we obtain that  for any fixed arbitrarily small constant $\epsilon >0$, this sequence changes its sign at least $x^{\frac{1}{5}-\epsilon}$ many times in the interval $(x,2x]$ for sufficiently large $x$. In this work, we obtain an improved result.
\end{rmk}

Through out the paper, $\epsilon$ denotes an arbitrary small positive constant but not necessarily the same one at each place.

\section{Key ingredients of Proof. }

The partial sums defined in \eqref{Sum-PIBQF} can be expressed in terms of known arithmetical functions, i.e.,
\begin{equation}
\begin{split}
S_{r}(f, D; x ) &=  \sum_{n \le x } \left( (\lambda_{f}(n))^{r} \left(\sum_{ n= Q(\underline{x}) } 1\right) \right) =  \sum_{n \le x } (\lambda_{f}(n))^{r} r_{Q}(n)\\
\end{split}
\end{equation}
 where $r_{Q}(n)$ denotes the number of representations of $n$ by the quadratic from $Q(\underline{x}) $.  We define the following Dirichlet series associated to the sum $S_{r}(f, D; x ) $: 
\begin{equation}\label{R-L-function}
\begin{split}
R_{r}(f, D; s ) &=  \sum_{n \ge 1 } \frac{(\lambda_{f}(n))^{r} r_{Q}(n)}{n^{s}},  \qquad  \qquad  \Re(s) >1.\\
\end{split}
\end{equation}
In \cite{Zhai},  Zhai has given the decomposition of $R_{r}(f, D; s )$ into well-known $L$- functions when $Q(\underline{x})= x_{1}^2+x_{2}^2$   (see \cite[Lemma 2.1]{Zhai}). We also use the same decomposition method to obtain the decomposition of $R_{r}(f, D; s )$ into well-known $L$-functions. We define the following $L$-functions associated to a normalised Hecke eigenform $f = \displaystyle{\sum_{n=1}^\infty \lambda_f(n)n^{\frac{k-1}{2}}q^{n}} \in S_{k}(SL_2(\mathbb{Z})).$ 
 The Hecke $L$-function  associated to $f$ is given by  
\begin{equation}
\begin{split}
L(s, f) &:= \sum_{n \ge 1} \frac{\lambda_{f}(n)}{n^{s}} = \prod_{p-{\rm prime}} \left(1-\frac{\alpha_p}{p^s}\right)^{-1}\left(1-\frac{\beta_p}{p^{s}}\right)^{-1}, \qquad \Re(s) >1,
\end{split}
\end{equation}
where $\alpha_p+\beta_p=\lambda_g(p)$ and $\alpha_p\beta_p=1.$  The Hecke $L$-function satisfies a nice functional equation and it has analytic continuation to whole $\mathbb{C}$-plane \cite[Section 7.2]{Iwaniec}.

\smallskip
For $m \ge 2 ,$ the $m^{th}$ symmetric power $L$-function is defined by
\begin{equation}\label{Symf}
\begin{split}
L(s,sym^{m}f)&: = \prod_{p-{\rm prime}} \prod_{j=0}^{m} \left(1-{\alpha_p}^{m-j}{\beta_p}^{j} {p^{-s}}\right)^{-1} 
                      =\sum_{n=1}^\infty \frac{\lambda_{sym^{m}f}(n)}{n^s},
\end{split}
\end{equation}
where ${\lambda_{sym^{m}f}(n)}$ is multiplicative function. From Deligne's bound, we have 
$$
|{\lambda_{sym^{m}f}(n)}| \le d_{m+1}(n) \ll_{\epsilon} n^{\epsilon}
$$
for any real number $\epsilon >0$ and $d_{m}(n)$ denotes the number of $m$ factors of $n$. 
\smallskip

For $M, N \ge 1 ,$  the Rankin-Selberg $L$- function associated to ${sym^{M}f}$ and ${sym^{N}f}$ ( given in terms of Euler product ) is defined as follows:
{\footnotesize
\begin{equation}\label{SymM-SymN}
\begin{split}
L(s,sym^{M}f \times sym^{N}f )&:= \prod_{p-{\rm prime}} \prod_{j=0}^{M} \prod_{l=0}^{N} \left(1-{\alpha_p}^{M-j}{\beta_p}^{j} {\alpha_p}^{N-l}{\beta_p}^{l} {p^{-s}}\right)^{-1} 
                      =\sum_{n=1}^\infty \frac{\lambda_{sym^{M}f \times sym^{N}f}(n)}{n^s},
\end{split}
\end{equation}
}
for $\Re(s) >1$ where ${\lambda_{sym^{M}f \times sym^{N}f}(n)}$ is also a multiplicative function and  from Deligne's bound, we have 
 $$|{\lambda_{sym^{M}f \times sym^{N}f}(n)}| \le d_{(M+1)(N+1)}(n) \ll_{\epsilon} n^{\epsilon}$$
for any real number $\epsilon >0$.
We write the symmetric power $L$-function and Rankin-Selberg $L$-function associated to symmetric power functions  as an Euler product, i.e., 
\begin{equation*}
\begin{split}
L(s,sym^{m}f) &= \prod_{p-{\rm prime}} \left(1+ \frac{\lambda_{sym^{m}f}(p)}{p^s} + \frac{\lambda_{sym^{m}f}(p^2)}{p^{2s}} \cdots \right)  \quad {\rm and }\\
L(s,sym^{M}f \times sym^{N}f) &= \prod_{p-{\rm prime}} \left(1+ \frac{\lambda_{sym^{M}f \times sym^{N}f}(p)}{p^s}  +  \frac{\lambda_{sym^{M}f \times sym^{N}f}(p^2)}{p^{2s}} \cdots \right). \\
\end{split}
\end{equation*}
where  the arithmetical functions ${\lambda_{sym^{m}f}(n)}$ and ${\lambda_{sym^{M}f \times sym^{N}f}(n)}$ are given by 
\begin{equation}\label{symFC}
\begin{split}
{\lambda_{sym^{m}f}(p)}  & = \sum_{j =0}^{m} {\alpha_p}^{j} {\beta_p}^{m-j} \quad  {~~~~~\rm and~~~~~~~} \quad 
{\lambda_{sym^{M}f \times sym^{N}f}(p)}  =  \sum_{j =0}^{M}  \sum_{l =0}^{N} {\alpha_p}^{j} {\beta_p}^{M-j} {\alpha_p}^{l} {\beta_p}^{N-j}. \\
\end{split}
\end{equation}

\begin{rmk}
For a classical holomorphic Hecke eigenform $f$, J. Cogdell and P. Michel \cite{Cog-Mic}  have given the explicit description of analytic continuation and functional equation  for the function $L(s,sym^{m}f)$, $2\le m \le 4$. Y. K. Lau and J. Wu \cite{Lau-Wu} have obtained functional equation and meromorphic continuation  for the $L$- functions  $L(s,sym^{M}f \times sym^{N}f)$ where $2 \le  M \le 4$ and $0 \le N \le 4$  when $f$ is classical holomorphic cusp form. 
\end{rmk}

\smallskip
We assume the following convention:
\begin{equation*}
\begin{split}
\begin{cases}
L(s, sym^{0}f) &= \zeta(s), \qquad \quad L(s, sym^{0}f \times \chi ) = L(s, \chi), \quad L(s, sym^{1}f  ) = L(s, f), \\
  L(s, sym^{1}f \times \chi ) & = L(s, f \times \chi), \qquad L(s, sym^{M}f \times sym^{0}f) = L(s, sym^{M}f ). \\
  \end{cases}
\end{split} 
\end{equation*}
where $\zeta(s) = \displaystyle{\sum_{n \ge 1} n^{-s}}$ and $L(s, \chi) = \displaystyle{\sum_{n \ge 1}\chi(n) n^{-s}}$ are the Riemann zeta function and Dirichlet $L$- function respectively. Now, we state the decomposition of $R_{r}(f, D; s )$, $1 \le r \le 8 $ into well known $L$-functions.

 \begin{lem}\label{L-fun-decomposition}
For  $ 1\le r \le 8,$ we have 
 \begin{equation} \label{DecompositionL}
\begin{split}
 R_{r}(f, D; s ) & = L_{r}(s) \times U_{r}(s), \quad \quad {\rm where } \\
 \end{split}
\end{equation}
 \begin{equation*}
\begin{split}
L_{1}(s) & =  L(s, f) L(s, f\times \chi_{D}),          \\
L_{2}(s) & =   \zeta(s) L(s, \chi_{D}) L(s, sym^{2}f)  L(s, sym^{2}f \times \chi_{D}),       \\
L_{3}(s) & =    L(s, f)^{2} L(s, f\times \chi_{D}) ^{2} L(s, sym^{3}f)  L(s, sym^{3}f \times \chi_{D}) ,       \\
L_{4}(s) & =     \zeta(s)^{2} L(s, \chi_{D})^{2} L(s, sym^{2}f) ^{3} L(s, sym^{2}f \times \chi_{D})^{3} L(s, sym^{4}f)  L(s, sym^{4}f \times \chi_{D}),       \\
L_{5}(s) & =    L(s, f)^{5} L(s, f \times \chi_{D})^{5} L(s, sym^{3}f)^{3} L(s, sym^{3}f \times \chi_{D})^{3} \\  
&  \qquad  \times L(s, sym^{4}f \times f)  L(s, sym^{4}f \times f \times \chi_{D}),  \\ 
\end{split}
\end{equation*}
 \begin{equation*}
\begin{split}
L_{6}(s) & =      \zeta(s)^{5} L(s, \chi_{D})^{5} L(s, sym^{2}f)^{8}  L(s, sym^{2}f \times \chi_{D})^{8} L(s, sym^{4}f) ^{4} L(s, sym^{4}f \times \chi_{D})^{4}  \\ &~~~~~~~~~~ \times   L(s, sym^{4}f \times sym^{2}f)  L(s, sym^{4}f \times sym^{2}f \times \chi_{D}) ,       \\
L_{7}(s) & =   L(s, f)^{13} L(s, f\times \chi_{D})^{13}  L(s, sym^{3}f)^{8}  L(s, sym^{3}f \times \chi_{D})^{8} L(s, sym^{4}f \times f)^{5} \\ &~~~~~~~~~ \times  L(s, sym^{4}f \times f \times \chi_{D})^{5}  L(s, sym^{4}f \times sym^{3}f)  L(s, sym^{4}f \times sym^{3}f \times \chi_{D})  \quad \quad and \\
L_{8}(s) & =       \zeta(s)^{13} L(s, \chi_{D})^{13} L(s, sym^{2}f)^{21}  L(s, sym^{2}f \times \chi_{D})^{21} L(s, sym^{4}f) ^{13} L(s, sym^{4}f \times \chi_{D})^{13} \\ ~~~~~~~~~& \times   L(s, sym^{4}f \times sym^{2}f)^{6}  L(s, sym^{4}f \times sym^{2}f \times \chi_{D})^{6}  \\ &~~~~~~~~ \times L(s, sym^{4}f \times sym^{4}f)  L(s, sym^{4}f \times sym^{4}f \times \chi_{D})        \\
\end{split}
\end{equation*}
where $\chi_{D}$ is the Dirichlet character modulo $|D|$ and $U_{r}(s)$ is a Dirichlet series which converges absolutely and uniformly for $\Re(s)>\frac{1}{2}$ and $U_{r}(s) \neq 0$ for $\Re(s)=1.$
\end{lem}


 \subsection{Proof of \lemref{L-fun-decomposition}}

Let   $T_{n}(x)$ (respectively $T_{m}(x) \times T_{n}(x)$) be the polynomial which gives trace of $n^{th}$ symmetric power of an element (respectively trace of Rankin - Selberg convolution of $m^{th}$ symmetric power and $n^{th}$ symmetric power functions) of  $Sl_{2}(\mathbb{C})$) whose trace is $x$ \cite[Proof of Lemma 2.1]{Lau-Lu-Wu} . Moreover, $T_{m}(x) \times T_{n}(x) = T_{m}(x)T_{n}(x)$. 

Let $U_{r}$ be the $r^{th}$ chebyshev polynomial of second kind given by; $U_{r}(cos ~\theta) = \frac{sin((r +1)\theta) }{sin(\theta)}$. Then  $T_{r}(x) = U_{r}(x/2)$. 
From Deligne's estimate, we write $\lambda_{f}(p) = 2 cos ~\theta$ and  $\lambda_{f}(p^{r}) = T_{r} (2 cos ~\theta) =U_{r} ( cos ~\theta) $. Then the expression for the $j^{th}$- symmetric power Fourier coefficient $\lambda_{sym^{j}f}(p)$ is given in term of $\lambda_{f}(p)  $ as  in \cite[Proof of lemma 2.2]{G. Lu1}.  From \eqref{symFC}, We write $\lambda_{f}(p) = \alpha_{p} + \beta_{p}$ and give a proof of decomposition in case of $l = 2$. Similar argument are used to prove other cases. Let us write $r^{*}_{Q}(n) = r_{Q}(n) / w_{D}  = \displaystyle{\sum_{d \vert n} \chi_{D}(d)}.$ The function $r^{*}_{Q}(n)$ is a multiplicative function. We know that  $\lambda_{f}(n)$ is a multiplicative function. This gives that the function $ R_{2}(f, D; s)$ is given as a Euler product, i.e.,
 \begin{equation*}
\begin{split}
R_{2}(f, D; s) = w_{D} \sum_{n \ge 1 } \frac{(\lambda_{f}(n))^{2} r^{*}_{Q}(n)}{n^{s}} = w_{D} \displaystyle{\prod_{p} \left(1+ \frac{(\lambda_{f}(p))^{2} r^{*}_{Q}(p)}{p^s} + \frac{(\lambda_{f}(p^2))^{2} r^{*}_{Q}(p^2)}{p^{2s}} + \cdots \right)} .
 \end{split}
\end{equation*}
Moreover,
\begin{equation*}
\begin{split}
{\lambda_{f}(p)}^{2} r^{*}_{Q}(p) & =  (\alpha_{p} + \beta_{p})^{2} (1 + \chi_{D}(p))  = ({\alpha_{p}}^{2} + {\beta_{p}}^{2} +2) (1 + \chi_{D}(p)) \\
& =  (\lambda_{sym^{2}f}(p)(1 + \chi_{D}(p))  + (1 + \chi_{D}(p))  \\
& = 1 + \chi_{D}(p) + \lambda_{sym^{2}f}(p) + \lambda_{sym^{2}f}(p)\chi_{D}(p) \\
\end{split}
\end{equation*}
as ${\alpha_{p}}  {\beta_{p}} = 1$ and  $\lambda_{sym^{2}f}(p) =  {\alpha_{p}}^{2} + {\beta_{p}}^{2} +1$ from \eqref{symFC}. For $\Re(s) >1$, we express the function
 $\zeta(s) L(s, \chi_{D}) L(s, sym^{2}f)  L(s, sym^{2}f \times \chi_{D})$
 as a Euler product of the form  
 \begin{equation*}
\begin{split}
 \displaystyle{\prod_{p} \left(1+ \frac{A(p)}{p^s} + \frac{A(p^{2})}{p^{2s}} + \cdots \right) }, \quad {\rm where } \quad A(p) = {\lambda_{f}(p)}^{2} r^{*}_{Q}(p).
 \end{split}
\end{equation*}
  So, we  have 
$$ R_{2}(f, D; s) = \zeta(s) L(s, \chi_{D}) L(s, sym^{2}f)  L(s, sym^{2}f \times \chi_{D}) \times U_{2}(s),  $$
 where the function $U_{2}(s)$ is given by the Eular product  \\
  $ w_{D} \displaystyle{\prod_{p} \left(1  + \frac{{(\lambda_{f}(p^{2})}^{2} r^{*}_{Q}(p^{2})- A(p^{2}))}{p^{2s}} + \cdots \right) }$ which converges absolutely for $\Re(s) > \frac{1}{2}$ and non-zero for $\Re(s) =1.$

\subsection{Sub-convexity bound, convexity bound and mean square integral estimates of the general $L$-functions}

\begin{lem}\label{Riemann Zeta }
Let $\zeta(s)= \displaystyle{\sum_{n\ge 1} \frac{1}{n^{s}}}$ be the Riemann zeta function. Then for any $\epsilon >0$, we have 
\begin{equation}
\begin{split}
\zeta(\sigma+it) &\ll_{\epsilon} (1+|t|)^{\frac{1}{3}(1-\sigma)+\epsilon}
\end{split}
\end{equation}
uniformly for $\frac{1}{2} \le \sigma \le 1$ and $|t| \ge 1$.
\end{lem}
\begin{proof}
 An estimate in \cite[Theorem 5.5]{ECT} (due to G. H. Hardy and J. E. Littlewood)  and Phragmen - Lindel\"{o}f convexity principle gives the required result.
\end{proof}

\begin{lem}\label{Dirichlet L function }
Let $\L(s, \chi)= \displaystyle{\sum_{n\ge 1} \frac{\chi(n)}{n^{s}}}$ be the Dirichlet $L$-function where $\chi$ is a Dirichlet character modulo N . Then for any $\epsilon >0$, we have 
\begin{equation}
\begin{split}
\L(\sigma+it, \chi) &\ll_{\epsilon, N} (1+|t|)^{\frac{1}{3}(1-\sigma)+\epsilon}
\end{split}
\end{equation}
uniformly for $\frac{1}{2} \le \sigma \le 1$ and $|t| \ge 1$.
\end{lem}

\begin{proof}
From  \cite[eq. (1.1)]{HBrown} and Phragmen - Lindel\"{o}f convexity principle, we have the required result.
\end{proof}

\begin{lem}\label{Modular L function }
For any $\epsilon >0$, the sub-convexity bound of Hecke $L$- function is given by:
\begin{equation}
\begin{split}
\L(\sigma+it, f) &\ll_{\epsilon} (1+|t|)^{\frac{2}{3}(1-\sigma)+\epsilon}
\end{split}
\end{equation}
uniformly for $\frac{1}{2} \le \sigma \le 1$ and $|t| \ge 1$.
\end{lem}

\begin{proof}
Proof follows from standard argument of Phragmen - Lindel\"{o}f convexity principle and a result of A. Good  \cite[Corollary]{AGood}.
\end{proof}


\begin{lem}
Let  $ m, M = 2, 3, 4$ and  $0 \le N \le M$.  For any arbitrarily small $\epsilon >0$, we have 
\begin{equation}
\begin{split}
L(\sigma+it, sym^{m}f) &\ll_{\epsilon} (1+|t|)^{\frac{m+1}{2}(1-\sigma)+\epsilon}\\
and \qquad  \quad L(\sigma+it, sym^{M}f \times sym^{N}f) &\ll_{\epsilon} (1+|t|)^{\frac{(M+1)(N+1)}{2}(1-\sigma)+\epsilon}
\end{split}
\end{equation}
uniformly for $\frac{1}{2} \le \sigma \le 1$ and $|t| \ge 1$.
\end{lem}

\begin{proof}
We refer to \cite[Lemma 2.5]{ G. Lu} and \cite[Section 2.3]{Lau-Lu-Wu}.
\end{proof}

\smallskip

Let  $l := (l_{1}, l_{2}, \cdots l_{u})$,  $M := (M_{1}, M_{2}, \cdots M_{u})$ and  $N := (N_{1}, N_{2}, \cdots N_{u})$ with $l_{\nu} \in {\mathbb N}$,  $ 1 \le M_{\nu} \le 4$ and $ 0 \le N_{\nu} \le M_{\nu}$. Following \cite{Lau-Lu-Wu},
 we define the following general $L$-functions :
\begin{equation}\label{Gen-L-fun}
\begin{split}
L_{M, N}^{l} (s) : = \prod_{\nu =1 }^{u} (L(s, sym^{M_{\nu}}f \times sym^{N_{\nu}}f))^{l_{\nu}},
\end{split}
\end{equation}
  $L_{M, N}^{l} (s)$ is a general  $L$- function of degree $\eta = \sum_{\nu =1}^{u} l_{\nu}(M_{\nu} +1) (N_{\nu} +1). $
  
  \begin{lem}\label{General L fun-Con } \cite[pp. 100]{HIwaniec}
Let $L(s,F)$ be an $L$- function of degree $m \ge 2,$ i.e.
\begin{equation}
\begin{split}
L(s, F) & = \sum_{n \ge 1} \frac{\lambda_{F}(n)}{n^{s}} = \prod_{p-{\rm prime}} \prod_{j= 1}^{m} \left(1-\frac{\alpha_{p, f, j}}{p^s}\right)^{-1},
\end{split}
\end{equation}
where $\alpha_{p, f, j}$, $1 \le j \le m$; are the local parameter of $L(F, s)$ at prime $p$ and $\lambda_{F}(n) = O(n^{\epsilon})$ for any $\epsilon>0.$ We assume that the series and Euler product converge absolutely for $\Re(s) > 1$ and $L(s, F)$ is an entire function except possibly for pole at $s = 1$ of order $r$ and satisfies a nice functional equation $(s \rightarrow 1-s)$. Then for any $\epsilon >0$, we have 
\begin{equation}
\begin{split}
\left( \frac{s-1}{s+1}\right)^{r}L(\sigma+it, F) &\ll_{\epsilon} (1+|t|)^{\frac{m}{2}(1-\sigma)+\epsilon}
\end{split}
\end{equation}
uniformly for $0 \le \sigma \le 1$.  and $|t| \ge 1$ where $s =\sigma+it$.
\end{lem}

\begin{lem}\label{Mean-value} \cite[Lemma 2.6]{G. Lu1}
Let $L(s,F)$ be an $L$-function of degree $m \ge 2$. Then for $T\ge 1$, We have 
\begin{equation}
\begin{split}
\int_{T}^{2T} \left|L \left(\frac{1}{2} + \epsilon+it, F\right)\right|^{2} dt &\ll T^{\frac{m}{2}+\epsilon}
\end{split}
\end{equation}

\end{lem}

 
 
 \smallskip
 The partial sum  $\sum\limits_{n \ge x } a(n)$ of arithmetical function $a(n)$ is related with its Dirichlet series  $D(s) = \sum\limits_{n \ge 1 } \frac{a_{n}}{n^{s}}$.  It  is given by following Lemma.
 
 \begin{lem}\label{perron-formula}  {\bf (Truncated Perron's formula)} \cite[ pp. 67]{MRam} $:$
Let $D(s) = \sum\limits_{n \ge 1 } \frac{a_{n}}{n^{s}}$ be a Dirichlet series. 
Let $a_n = O(n^{\epsilon})$ for an arbitrary small $\epsilon > 0$.  Then for  a positive non-integral $x$, the truncated Perron' s formula is given by 
\begin{equation}\label{MRam-Perron}
\sum_{n \le x} a_{n} =  \frac{1}{2 \pi i}\int_{\sigma -i T}^{\sigma +i T}  D(s) \frac{x^s}{s} ds + O\left( \frac{x^{\sigma+\epsilon}}{T} \right)
\end{equation}
where $\sigma =\Re(s) >1$ is a real number. This result also holds when $x$ is a positive integer.
\end{lem}

Now, we state and prove a result on the sign change which is a generalisation of  the  result of Meher - Murty \cite[Theorem 1.1]{ram-meher}.
\begin{lem} \label{lalit-sign change} 
Let $\{a(n)\}_{n \in \mathbb{N}}$ be a sequence of real numbers and $g  : {\textbf{S}}^{r} \rightarrow \mathbb{N}$ ba a function such that $g^{-1}(n)$ is a finite set for each $n$ where  $\textbf{S}$ is a countable set and $r \ge 1.$ Suppose that
\begin{enumerate}
\item  $a({n}) = O(n^{\alpha})$ for each $n,$
\item  $ \displaystyle{\sum_{ \underline{s} \in {\textbf{S}}^{r} \atop g(\underline{s}) \le x} a({g(\underline{s}))} = O(x^{\beta})},$
\item  $\displaystyle{\sum_{ \underline{s} \in {\textbf{S}}^{r} \atop g(\underline{s}) \le x} a^{2}({g(\underline{s}))} = C x +  O(x^{\gamma})}$. 
\end{enumerate}
with $\alpha, \beta, \gamma > 0$ {~~~~and ~~~~} $C > 0$. Let  $\alpha + \beta < 1.$  Then for any $\delta$ satisfying max $\{\alpha + \beta, \gamma \} < \delta < 1$,
the sequence  $\{a(g(\underline{s}))\}_{\underline{s} \in \textbf{S}^{r}}$ has at least one sign change at $\underline{s} \in {\textbf{S}}^{r}$ with $ g(\underline{s}) \in (x, x+x^{\delta}].$ Consequently, there are at least $x^{1-\delta}$ sign changes for $ g(\underline{s}) \in (x, 2x],$ for sufficiently large $x$.
\end{lem}

\textbf{Proof of \lemref{lalit-sign change}:}
We assume that the sequence  $\{a(g(\underline{s}))\}_{\underline{s} \in \textbf{S}^{r}}$ has constant sign (say positive) for each $\underline{s} \in {\textbf{S}}^{r}$ with $ g(\underline{s}) \in (x, x+x^{\delta}].$   Then from conditions $(1)$ and $(2),$ we have 
\begin{equation}\label{contra1}
\begin{split}
\displaystyle{\sum_{ \underline{s} \in {\textbf{S}}^{r} \atop x\le g(\underline{s}) \le x+ x^{\delta}} a({g(\underline{s}))}^2} \ll x^{\alpha} \displaystyle{\sum_{ \underline{s} \in {\textbf{S}}^{r} \atop x\le g(\underline{s}) \le x+ x^{\delta}} a({g(\underline{s}))}}\ll x^{\alpha+ \beta}
\end{split}
\end{equation}
and the condition $(3)$ gives the following:
\begin{equation}\label{contra2}
\begin{split}
\displaystyle{\sum_{ \underline{s} \in {\textbf{S}}^{r} \atop x\le g(\underline{s}) \le x+ x^{\delta}} a({g(\underline{s}))}^2} 
= C x^{\delta} + O(x^{\gamma}) \gg x^{\delta}.
\end{split}
\end{equation}
From \eqref{contra1} and \eqref{contra2}, we  have 
$$x^{\delta} \ll \displaystyle{\sum_{ \underline{s} \in {\textbf{S}}^{r} \atop x\le g(\underline{s}) \le x+ x^{\delta}} a({g(\underline{s}))}^2}  \ll x^{\alpha+ \beta},$$
which is a contradiction as ${\alpha+ \beta} < \delta.$ Thus, we have at least one sign change of the sequence $\{a(g(\underline{s}))\}_{\underline{s} \in \textbf{S}^{r}}$ for some $\underline{s} \in {\textbf{S}}^{r}$ with $ g(\underline{s}) \in (x, x+x^{\delta}].$

 \section{Proof of  results.} 

\subsection{Proof of \thmref{Estimate-BFQ}}
In order to obtain an estimate of $S_{r}(f, D; x )$, we apply the truncated Perron's formula (\lemref{perron-formula}). 


We choose $\epsilon >0$ for which  $\lambda_{f}(n) \ll_{\epsilon} n^{\epsilon/32}$ and $r_{Q}(n) \ll_{\epsilon} n^{\epsilon/2}$ so that we have $$(\lambda_{f}(n))^{r} r_{Q}(n) \ll_{\epsilon} n^{\epsilon(r/32 +1/2 )} \ll_{\epsilon} n^{\epsilon /R} \ll_{\epsilon} n^{\epsilon}$$ 
where $R >1.$ This gives $(\lambda_{f}(n))^{r} r_{Q}(n) \ll_{\epsilon} n^{\epsilon}$ for each $r \ge 1.$ Now, we apply the truncated Perron's formula (\lemref{perron-formula}) by taking $\sigma = 1+\epsilon$ to obtain an estimate for $S_{r}(f, D; x )$  for each $r$ with  $1 \le r \le 8$, i.e.,
\begin{equation}
\begin{split}
S_{r}(f, D; x ) = \sum_{n \le x } (\lambda_{f}(n))^{r} r_{Q}(n) = \frac{1}{2 \pi i}\int_{1+\epsilon -i T}^{1 + \epsilon +i T}  R_{r}(f, D; s ) \frac{x^s}{s} ds + O\left( \frac{x^{1+ 2\epsilon}}{T} \right), \\ 
\end{split}
\end{equation}
where $R_{r}(f, D; s ) $ is given in \eqref{R-L-function}. Now, we shift the line of integration from $1 + \epsilon $ to $\frac{1}{2} + \epsilon$ and apply the Cauchy residue theorem to get: 
 \begin{equation*}
 \begin{split}
S_{r}(f, D; x ) & = \underset{s= 1}{\rm Res} \left(R_{r}(f, D; s ) \frac{x^{s}}{s} \right) \\   & +  \frac{1}{2 \pi i} \left \{ \int_{\frac{1}{2} + \epsilon -i T}^{\frac{1}{2} +\epsilon +i T} + \int_{1+\epsilon - i T}^{\frac{1}{2} + \epsilon -i T} + \int_{ \frac{1}{2} + \epsilon + i T}^{1+\epsilon + i T} \right \} R_{r}(f, D; s ) \frac{x^s}{s} ds + O\left( \frac{x^{1+2\epsilon}}{T}\right).  \\
 \end{split}
\end{equation*} 
where the first term of right-hand side does not appear when $r$ is odd. In this case ($r$-odd), the function $R_{r}(f, D; s )$ does not have a pole at $s=1.$ We define the following integrals:
\begin{equation}
\begin{split}
V_{r} & := \frac{1}{2 \pi i}  \int_{\frac{1}{2} + \epsilon -i T}^{\frac{1}{2} +\epsilon +i T} R_{r}(f, D; s ) \frac{x^s}{s} ds \\  {~~~~~~~~\rm and ~~~~~~~~~~} \qquad \qquad 
H_{r}  & :=  \left\{ \int_{1+\epsilon - i T}^{\frac{1}{2} + \epsilon -i T} + \int_{ \frac{1}{2} + \epsilon + i T}^{1+\epsilon + i T} \right \} R_{r}(f, D; s ) \frac{x^s}{s} ds.  
 \end{split}
\end{equation}
For each $r$ with $ 1 \le r \le 8.$ We shall obtain an upper bound for $H_{r}$ and $V_{r}$ using the following arguments: 
\begin{equation*}
 \begin{split}
|H_{r} | & =  \left| \left\{ \int_{1+\epsilon - i T}^{\frac{1}{2} + \epsilon -i T} + \int_{ \frac{1}{2} + \epsilon + i T}^{1+\epsilon + i T} \right \} R_{r}(f, D; s ) \frac{x^s}{s} ds \right|  \\
            &  = \left|  \int_{1+\epsilon }^{\frac{1}{2} + \epsilon}  R_{r}(f, D; \sigma -i T ) \frac{x^{\sigma -i T}}{\sigma -i T} d\sigma  + \int_{ \frac{1}{2} + \epsilon }^{1+\epsilon}  R_{r}(f, D; \sigma +i T ) \frac{x^{\sigma +i T}}{\sigma +i T} d\sigma \right|  \\
            & \ll \int_{1+\epsilon }^{\frac{1}{2} + \epsilon}  |R_{r}(f, D; \sigma +i T )| \frac{x^{\sigma}}{T} d\sigma 
             \ll  \underset{\frac{1}{2}+ \epsilon \le \sigma \le 1+ \epsilon} {max} \left (\frac{x^{\sigma}}{T} |R_{r}(f, D;  \sigma + iT)| \right ). \\
 \end{split}
\end{equation*}
The last estimate is obtained by substituting the convexity or sub-convexity bound of $R_{r}(f, D; s )$ and observing that the    integrand is a convex function in $\sigma.$
 \begin{equation*} 
 \begin{split}
V_{r} & =  \frac{1}{2 \pi i}  \int_{\frac{1}{2}+\epsilon -i T}^{\frac{1}{2}+\epsilon +i T} R_{r}(f, D; s ) \frac{x^s}{s} ds =  \frac{1}{2 \pi i}  \int_{- T}^{ T} R_{r}(f, D; \frac{1}{2}+\epsilon + i t) \frac{x^{\frac{1}{2}+\epsilon + i t}}{(\frac{1}{2}+\epsilon + i t)} dt  \\
|V_{r}| 
&  \ll x^{\frac{1}{2}+\epsilon} \int_{0}^{ T} \frac{|R_{r}(f, D; \frac{1}{2}+\epsilon + i t)|}{|\frac{1}{2}+\epsilon + i t|} dt  \ll x^{\frac{1}{2}+\epsilon} \left \{ \int_{0}^{ 1} + \int_{1}^{ T} \right \} \frac{|R_{r}(f, D; \frac{1}{2}+\epsilon + i t)|}{|\frac{1}{2}+\epsilon + i t|} dt \\  & \ll x^{\frac{1}{2}+\epsilon} \left( 1 + \int_{1}^{ T} \frac{1}{t} |R_{r}(f, D; \frac{1}{2}+\epsilon + i t)| dt \right) \ll x^{\frac{1}{2}+\epsilon} \left( 1+ \int_{1}^{ T} \frac{1}{t} |R_{r}(f, D; \frac{1}{2}+\epsilon + i t)| dt \right). \\
\end{split}
\end{equation*} 
Since the bound for  ${|R_{r}(f, D; \frac{1}{2}+\epsilon + i t)|}$ is even function in $t$. We obtain an estimate for $|V_{r}|$ using the following arguments: \\
 Let $f(t) = R_{r}(f, D; \frac{1}{2}+ \epsilon + i t).$ Then  for some $c>0,$
\begin{equation}\label{stand-argument}
\begin{split}
 \int_{1}^{T} \frac{|f(t)|}{|t|} dt  &
\le \sum_{i=0}^{c \log T} \int_{\frac{T}{2^{i+1}}}^{\frac{T}{2^i}} \frac{|f(t)|}{|t|} dt   
 \le  \underset{i}{max} \left( \frac{1}{{\frac{T}{2^{i+1}}}} \int_{\frac{T}{2^{i+1}}}^{\frac{T}{2^i}} |f(t)| dt  \right) \times c \log T  \\
 &  \le  \underset{ 2 \le T_{1} \le T}{max} \left( \frac{2}{T_{1}} \int_{\frac{T_{1}}{2}}^{T_{1}} |f(t)| dt  \right) \times c \log T  
   \ll  \log T  \underset{ 2 \le T_{1} \le T}{max} \left( \frac{1}{T_{1}} \int_{\frac{T_{1}}{2}}^{T_{1}} |f(t)| dt  \right). \\
\end{split}
\end{equation}
So, we have 
 \begin{equation*} 
 \begin{split}
|V_{r}| 
           & \ll x^{\frac{1}{2}+\epsilon} \left( 1+\log T \underset{ 2 \le T_{1} \le T}{max} \left( \frac{1}{T_{1}} \int_{\frac{T_{1}}{2}}^{T_{1}}|                R_{r}(f, D; \frac{1}{2}+\epsilon + i t)| dt  \right) \right). \\
\end{split}
\end{equation*} 
Since $\log T \ll T^{\epsilon}$ for any $\epsilon>0$ and $1\le T\le x.$ So $\log T \ll x^{\epsilon}.$ Moreover, we use the decomposition $R_{r}(f, D; s) = L_{r}(s)U_{r}(s)$ from \lemref{L-fun-decomposition} and absolute convergence of  $U_{r}(s)$ for $\Re(s) > \frac{1}{2}$ to have 
\begin{equation} \label{Hor-Ver-Est}
 \begin{split}
 |V_{r}|  & \ll_{f, \epsilon} x^{\frac{1}{2}+ \epsilon} + x^{\frac{1}{2}+2 \epsilon} \underset{ 2 \le T_{1} \le T}{max} \left( \frac{1}{T_{1}} \int_{\frac{T_{1}}{2}}^{T_{1}} |L_{r}(\frac{1}{2}+\epsilon + i t)| dt  \right) \\
 {\rm and} \quad  \qquad
|H_{r} | & \ll  \underset{\frac{1}{2}+ \epsilon \le \sigma \le 1+ \epsilon} {max} \left (\frac{x^{\sigma}}{T} |L_{r}(\sigma + iT)| \right ). \\
 \end{split}
\end{equation}
We then apply the convexity or sub-convexity bound and  mean square integral estimate for the $L$-function to obtain an estimate of $S_{r}(f, D; x),$ i.e.
\begin{equation}\label{Estimates-SR}
 \begin{split}
S_{r}(f, D; x) & = \underset{s= 1}{\rm Res} \left(R_{r}(f, D; s ) \frac{x^{s}}{s} \right) +  |V_{r}| +  |H_{r}| + O\left( \frac{x^{1+     2 \epsilon}}{T} \right). \\   
 \end{split}
\end{equation} 


{\bf Case  r = 1 :}

From \lemref{L-fun-decomposition},  we have $L_{1}(s)  =  L(s, f) L(s, f\times \chi_{D})$.  $L_{1}(s)$ does not have a pole as $L(s, f)$ and $L(s, f\times \chi_{D})$ are entire functions. Now, we substitute $L_{1}(s) = L(s, f) L(s, f\times \chi_{D})$ to obtain an estimate for $H_{1}$ and $V_{1}$.
\begin{equation*} 
 \begin{split}
H_{1}  & =  \left\{ \int_{1+\epsilon - i T}^{\frac{1}{2} + \epsilon -i T} + \int_{ \frac{1}{2} + \epsilon + i T}^{1+\epsilon + i T} \right \} R_{1}(f, D; s ) \frac{x^s}{s} ds \\
 |H_{1}|  &  \ll \left|\int_{\frac{1}{2} + \epsilon + i T}^{1+\epsilon + i T} R_{1}(f, D; s ) \frac{x^s}{s} ds\right|    \ll  \underset{\frac{1}{2}+ \epsilon \le \sigma \le 1+ \epsilon} {max} \left (\frac{x^{\sigma}}{T} |R_{1}(f, D; \sigma + iT)| \right ) \\
        &  \ll  \underset{\frac{1}{2}+ \epsilon \le \sigma \le 1+ \epsilon} {max} \left (\frac{x^{\sigma}}{T} |L(\sigma+ i T, f) L(\sigma+ i T, f \times \chi_{D}) U_{1}(\sigma + i T)| \right) \\
        &  \ll  \underset{\frac{1}{2}+ \epsilon \le \sigma \le 1+ \epsilon} {max} \left (\frac{x^{\sigma}}{T} T^{2\times \frac{2}{3} \times(1-\sigma)+\epsilon} \right)
         \ll \frac{x^{1+\epsilon}}{T} + x^{\frac{1}{2}+\epsilon} T^{\frac{-1}{3} + \epsilon} \ll \frac{x^{1+\epsilon}}{T} ,
 \end{split}
\end{equation*}
provided $T \le x^{3/4}.$ Here, we use the sub-convexity bound of the  Hecke $L$-function $L(s, f)$ and $L(s,  f \times \chi_{D})$  to obtain an estimate. An estimate of $V_{1}$ is obtained as follows:
 \begin{equation*} 
 \begin{split}
 |V_{1}| 
& \ll x^{\frac{1}{2}+\epsilon} \left( \int_{1}^{ T} \frac{1}{t} |L(\frac{1}{2}+\epsilon + i t, f) L(\frac{1}{2}+\epsilon + i t, f \times \chi_{D}) U_{1}(\frac{1}{2}+\epsilon + i t)| dt \right).
\end{split}
\end{equation*}
Using absolute convergence of $U_{1}(s)$ for $\Re(s) > \frac{1}{2}$ and sub-convexity bound, we have 
\begin{equation*} 
 \begin{split}
|V_{1}| & \ll  x^{\frac{1}{2}+ \epsilon}  \times  \int_{1}^{T} \frac{1}{t} (1+t)^{(2 \times \frac{2}{3} (1 - (1/2+\epsilon)) + \epsilon)} dt 
 \ll x^{\frac{1}{2}+ \epsilon} T^{\frac{2}{3}+\epsilon}.
\end{split}
\end{equation*}
Thus, we have 
 \begin{equation*}
 \begin{split}
S_{1}(f, D; x ) & \ll \frac{x^{1+\epsilon}}{T} + x^{\frac{1}{2}+\epsilon} T^{\frac{2}{3} + \epsilon} + \frac{x^{1+2\epsilon}}{T} . \\
       \end{split}
\end{equation*} 
Now, we choose $T = x^{\frac{3}{10}}$  to get 
\begin{equation*}
 \begin{split}
S_{1}(f, D; x ) & \ll x^{\frac{7}{10}+ 2\epsilon}. \\
       \end{split}
\end{equation*}

\smallskip
For  $r$ with $2 \le r \le 8$,  we use the following method as done in \cite{Zhai} to obtain bound for $V_{r}$.
  \begin{equation*} 
  \begin{split}
|V_{r}| &  \ll x^{\frac{1}{2}+\epsilon} + x^{\frac{1}{2}+2\epsilon} \underset{2 \le T_{1} \le T}{max}  \left( \frac{1}{T_{1}}  \int_{\frac{T_{1}}{2}}^{ T_{1}} \left|L_{r} \left(\frac{1}{2}+\epsilon + i t \right) \right| dt \right). \\
\end{split}
\end{equation*} 
Now, we break $L_{r}(s)$ into two well-known $L$-functions and apply the Cauchy- Schwarz inequality  to get 
\begin{equation}\label{Ver-Est}
 \begin{split}
|V_{r}| & \ll x^{\frac{1}{2}+\epsilon} + x^{\frac{1}{2}+2\epsilon} \underset{2 \le T_{1} \le T}{max} 
\begin{cases}
  \frac{1}{T_{1}}  \left( \int_{\frac{T_{1}}{2}}^{ T_{1}} \left|L_{r,1} \left(\frac{1}{2}+\epsilon + i t \right) \right|^{2} dt \right)^{\frac{1}{2}}  \\  \qquad  \times \left( \int_{\frac{T_{1}}{2}}^{ T_{1}} \left|L_{r, 2} \left(\frac{1}{2}+\epsilon + i t \right) \right|^{2} dt \right)^{\frac{1}{2}} . \\
 \end{cases}
\end{split}
\end{equation}
From \eqref{Hor-Ver-Est}, we have
\begin{equation}\label{Hor-Est}
 \begin{split}
 |H_{r} | & \ll  \underset{\frac{1}{2}+ \epsilon \le \sigma \le 1+ \epsilon} {max} \left (\frac{x^{\sigma}}{T} |L_{r}(\sigma + iT)|\right ).
\end{split}
\end{equation} 
 We apply the convexity bound, sub-convexity bound, and mean square integral estimate of the $L$-function to obtain the bounds for $H_{r}$ and $V_{r}$. Thus, we obtain the required estimate for $S_{r}(f, D; x )$.

\bigskip 
{\bf Case : r = 2 } 
 \begin{equation*}
\begin{split}
L_{2}(s) & =   \zeta(s) L(s, \chi_{D}) L(s, sym^{2}f)  L(s, sym^{2}f \times \chi_{D}).      \\
\end{split}
\end{equation*}
We have
\begin{equation*} 
 \begin{split}
|H_{2}|  & \ll  \underset{\frac{1}{2}+ \epsilon \le \sigma \le 1+ \epsilon} {max} \left (\frac{x^{\sigma}}{T} |\zeta(\sigma +iT) L(\sigma +iT, \chi_{D}) L(\sigma +iT, sym^{2}f)  L(\sigma +iT, sym^{2}f \times \chi_{D})| \right ) \\
 & \ll \frac{x^{1+\epsilon}}{T} + \left| \frac{x^{\sigma}}{T} \zeta(\sigma +iT) L(\sigma +iT, \chi_{D}) L(\sigma +iT, sym^{2}f)  L(\sigma +iT, sym^{2}f \times \chi_{D}) \right|_{\sigma = \frac{1}{2}+ \epsilon}\\
 & \ll \frac{x^{1+\epsilon}}{T} + \left (\frac{x^{\frac{1}{2}+ \epsilon}}{T}  T^{2 \times ( \frac{1}{6}+\epsilon)  + 2 \times ( \frac{3}{4}+\epsilon)  }\right) \ll \frac{x^{1+\epsilon}}{T} + x^{\frac{1}{2}+ \epsilon} T^{5/6+ \epsilon}  \\ 
 \end{split}
\end{equation*}
and
\begin{equation*}
 \begin{split}
 |V_{2}| & \ll x^{\frac{1}{2}+\epsilon} + x^{\frac{1}{2}+2\epsilon} \left( \int_{1}^{ T} \frac{1}{t} \left|L_{2} \left(\frac{1}{2}+\epsilon + i t \right) \right| dt \right) \\
V_{2}  & \ll  x^{\frac{1}{2}+\epsilon}   + x^{\frac{1}{2}+2\epsilon} \underset{2 \le T_{1} \le T}{max}  
\begin{cases}  \frac{1}{T_{1}}  \left( \int_{\frac{T_{1}}{2}}^{ T_{1}} \left|\zeta(1/2+ \epsilon +it )  L(1/2+ \epsilon +it , sym^{2}f)\right|^{2} dt \right)^{\frac{1}{2}}  \\ \times \left( \int_{\frac{T_{1}}{2}}^{ T_{1}} \left|L(1/2+ \epsilon +it , \chi_{D}) L(1/2+ \epsilon +it , sym^{2}f \times \chi_{D})\right|^{2} dt \right)^{\frac{1}{2}}  \\
\end{cases}
\end{split}
\end{equation*} 
\begin{equation*}
 \begin{split}
V_{2}  & \ll x^{\frac{1}{2}+\epsilon} \\
& \quad + x^{\frac{1}{2}+2\epsilon} \underset{2 \le T_{1} \le T}{max}  
\begin{cases}  \frac{1}{T_{1}}  \left(\underset{\frac{T_{1}}{2} \le t \le T_{1}}{max} |\zeta(1/2+ \epsilon +it ) |^{2} \int_{\frac{T_{1}}{2}}^{ T_{1}} \left| L(1/2+ \epsilon +it , sym^{2}f)\right|^{2} dt \right)^{\frac{1}{2}} \times \\ \left(\underset{\frac{T_{1}}{2} \le t \le T_{1}}{max} |L(1/2+ \epsilon +it , \chi_{D})|^{2} \int_{\frac{T_{1}}{2}}^{ T_{1}} \left| L(1/2+ \epsilon +it , sym^{2}f \times \chi_{D})\right|^{2} dt \right)^{\frac{1}{2}}.  \\
\end{cases}
\end{split}
\end{equation*} 
From  Lemmas  \ref{Riemann Zeta }, \ref{Dirichlet L function }, \ref{Modular L function }, \ref{General L fun-Con } and \ref{Mean-value}, we have
\begin{equation*}
 \begin{split}
V_{2}  & \ll x^{\frac{1}{2}+\epsilon} + x^{\frac{1}{2}+2\epsilon} \underset{2 \le T_{1} \le T}{max}  \left(
 \frac{1}{T_{1}}  \left( T_{1}^{\frac{2}{6}+\epsilon}  \times T_{1}^{\frac{3}{2}+\epsilon} \right)^{\frac{1}{2}} \times \left( T_{1}^{\frac{2}{6}+\epsilon} \times T_{1}^{\frac{3}{2}+\epsilon} \right)^{\frac{1}{2}}  \right) \\
 &  \ll x^{\frac{1}{2}+\epsilon} + x^{\frac{1}{2}+2\epsilon}  T^{\frac{5}{6}+\epsilon} \ll x^{\frac{1}{2}+2\epsilon}  T^{\frac{5}{6}+\epsilon}. \\
\end{split}
\end{equation*} 
So, we have 
 \begin{equation}
 \begin{split}
S_{2}(f, D; x ) & = C x P_{2}(\log x) +O\left( x^{\frac{1}{2}+ \epsilon} T^{\frac{5}{6}+\epsilon} +  \frac{x^{1+\epsilon}}{T}\right) + O\left( x^{\frac{1}{2}+2\epsilon} T^{\frac{5}{6} + \epsilon}\right) +O\left( \frac{x^{1+2\epsilon}}{T}\right).  \\
       \end{split}
\end{equation} 
Now, we choose $T = x^{\frac{3}{11}}$  to get 
\begin{equation*}
 \begin{split}
S_{2}(f, D; x ) & = C x P_{2}(\log x) + O \left( x^{\frac{8}{11}+ 2\epsilon} \right). \\
       \end{split}
\end{equation*}

\bigskip
For $r$ with $3 \le r \le 8$, the method to obtain estimates for $H_{r}$ and $V_{r}$ is exactly same as done in the case of $r=2$.  For each $r$ with $3 \le r \le 8$, estimate for $S_{r}(f, D; x )$ are exactly same as in \cite{Zhai}.  

{\bf Case : r = 3 }
 \begin{equation*}
\begin{split}
L_{3}(s) & =  L_{3,1}(s) \times L_{3,2}(s), \quad {\rm where} \quad 
\begin{cases}
L_{3,1}(s) =   L(s, f)^{2} L(s, sym^{3}f),     \\ 
L_{3,2}(s) =  L(s, f\times \chi_{D}) ^{2} L(s, sym^{3}f \times \chi_{D}).        \\
\end{cases}
\end{split}
\end{equation*}
Substituting these in \eqref{Ver-Est} and \eqref{Hor-Est} when $r =3$ and applying the convexity or sub-convexity bound and mean square estimate of the $L$-function, we have
\begin{equation*} 
 \begin{split}
|H_{3}|   & \ll \frac{x^{1+\epsilon}}{T} + \left (\frac{x^{\frac{1}{2}+ \epsilon}}{T}  T^{  4 \times ({\frac{1}{3}+\epsilon} )+ 2 \times ( \frac{4}{2.2}+\epsilon)  }\right) \ll \frac{x^{1+\epsilon}}{T} + x^{\frac{1}{2}+ \epsilon} T^{7/3+ \epsilon}  \\ 
 {\rm and}  \quad 
|V_{3}|  & \ll x^{\frac{1}{2}+\epsilon} + x^{\frac{1}{2}+2\epsilon} \underset{2 \le T_{1} \le T}{max}  \left(
 \frac{1}{T_{1}}  \left( T_{1}^{4 \times (\frac{1}{3}+\epsilon)}  \times T_{1}^{\frac{4}{2}+\epsilon} \right)^{\frac{1}{2}} \times \left( T_{1}^{4 \times (\frac{1}{3}+\epsilon)}  \times T_{1}^{\frac{4}{2}+\epsilon} \right)^{\frac{1}{2}}  \right) \\
 &\ll  x^{\frac{1}{2}+\epsilon} + x^{\frac{1}{2}+ 2\epsilon}  T^{\frac{7}{3}+\epsilon}  \ll x^{\frac{1}{2}+ 2\epsilon}  T^{\frac{7}{3}+\epsilon}.
 \end{split}
\end{equation*}

We substitute these bounds in \eqref{Estimates-SR} to get 
 \begin{equation*}
 \begin{split}
S_{3}(f, D; x ) & \ll  x^{\frac{1}{2}+ \epsilon} T^{\frac{7}{3}+\epsilon} +  \frac{x^{1+\epsilon}}{T} + x^{\frac{1}{2}+2\epsilon} T^{\frac{7}{3} + \epsilon} + \frac{x^{1+2\epsilon}}{T} . \\
       \end{split}
\end{equation*} 
Now, we choose $T = x^{\frac{3}{20}}$  to get 
\begin{equation*}
 \begin{split}
S_{3}(f, D; x ) & \ll  x^{\frac{17}{20}+ 2\epsilon}. \\
       \end{split}
\end{equation*}
 
{\bf Case : r = 4 }
 
 \begin{equation*}
\begin{split}
L_{4}(s) & =     \zeta(s)^{2} L(s, \chi_{D})^{2} L(s, sym^{2}f) ^{3} L(s, sym^{2}f \times \chi_{D})^{3} L(s, sym^{4}f)  L(s, sym^{4}f \times \chi_{D}).       \\
L_{4}(s) & =  L_{4,1}(s) \times L_{4,2}(s), 
\end{split}
\end{equation*}
 \begin{equation*}
\begin{split}
{\rm where} \qquad 
\begin{cases}
L_{4,1}(s) =   \zeta(s)^{2}  L(s, sym^{2}f) ^{3}  L(s, sym^{4}f),     \\ 
L_{4,2}(s) =  L(s, \chi_{D})^{2} L(s, sym^{2}f \times \chi_{D})^{3}  L(s, sym^{4}f \times \chi_{D}).     \\
\end{cases}
\end{split}
\end{equation*}
Substituting these in \eqref{Ver-Est} and \eqref{Hor-Est} when $r = 4$ and applying the convexity or sub-convexity bound and mean square estimate of the $L$-function, we have
 \begin{equation*} 
 \begin{split}
 |H_{4}| & \ll \frac{x^{1+\epsilon}}{T} + \left (\frac{x^{\frac{1}{2}+ \epsilon}}{T}  T^{2 \times ( \frac{1}{6}+\epsilon) + 2 \times \frac{3.3}{2} \times ( \frac{1}{2}+\epsilon)  + 2 \times \frac{5}{2} \times ( \frac{1}{2}+\epsilon)  }\right) \ll \frac{x^{1+\epsilon}}{T} + x^{\frac{1}{2}+ \epsilon}T^{20/3+ \epsilon} \\
  {\rm and } \quad 
 |V_{4}|  & \ll x^{\frac{1}{2}+\epsilon} + x^{\frac{1}{2}+2\epsilon} \underset{2 \le T_{1} \le T}{max}  \left(
 \frac{1}{T_{1}}  \left( T_{1}^{4 \times \frac{1}{6}+\epsilon}  \right) \times  \left( T_{1}^{\frac{14}{2}+\epsilon} \right)^{2 \times \frac{1}{2}} \right) \\
 & \ll x^{\frac{1}{2}+\epsilon} + x^{\frac{1}{2}+2\epsilon}  T^{\frac{20}{3}+\epsilon}  \ll x^{\frac{1}{2}+2\epsilon}  T^{\frac{20}{3}+\epsilon}.\\
 \end{split}
\end{equation*}
 Substituting these estimates in \eqref{Estimates-SR} when $r =4$, we get 
 \begin{equation*}
 \begin{split}
S_{4}(f, D; x ) & =  C x P_{2}(\log x) + O\left( x^{\frac{1}{2}+ \epsilon} T^{\frac{20}{3}+\epsilon} +  \frac{x^{1+\epsilon}}{T}\right) +  O\left( x^{\frac{1}{2}+2\epsilon} T^{\frac{20}{3} + \epsilon}\right) + O\left( \frac{x^{1+2\epsilon}}{T} \right). \\
       \end{split}
\end{equation*} 
Now, we choose $T = x^{\frac{3}{46}}$ to get 
\begin{equation*}
 \begin{split}
S_{4}(f, D; x ) & = C x P_{4}(\log x) +O \left( x^{\frac{43}{46}+ 2\epsilon} \right). \\
       \end{split}
\end{equation*}

{\bf Case : r = 5 }
 \begin{equation*}
\begin{split}
L_{5}(s) & =    L(s, f)^{5} L(s, f \times \chi_{D})^{5} L(s, sym^{3}f)^{3} L(s, sym^{3}f \times \chi_{D})^{3} \times \\ & \qquad  L(s, sym^{4}f \times f)  L(s, sym^{4}f \times f \times \chi_{D}) .  \\
L_{5}(s) & =  L_{5,1}(s) \times L_{5,2}(s), 
\end{split}
\end{equation*}
\begin{equation*}
\begin{split}
 {\rm where}  \qquad
\begin{cases}
L_{5,1}(s)  =   L(s, f)^{5}  L(s, sym^{3}f)^{3}  L(s, sym^{4}f \times f) ,  \\ 
 L_{5,1}(s) =   L(s, f \times \chi_{D})^{5} L(s, sym^{3}f \times \chi_{D})^{3}  L(s, sym^{4}f \times f \times \chi_{D}) .
   \end{cases}
\end{split}
\end{equation*}
Substituting these in \eqref{Ver-Est} and \eqref{Hor-Est} when $r =5$. Then we apply the convexity bound and mean square estimate to get  
\begin{equation*} 
 \begin{split}
 |H_{5}|& \ll \frac{x^{1+\epsilon}}{T} + \left (\frac{x^{\frac{1}{2}+ \epsilon}}{T}  T^{  10 \times ({\frac{1}{3}+\epsilon} )+ 6 \times \frac{4}{2 } \times ( \frac{1}{2}+\epsilon) +2 \times \frac{10}{2 } \times ( \frac{1}{2}+\epsilon)  }\right)  \ll \frac{x^{1+\epsilon}}{T} + x^{\frac{1}{2}+ \epsilon} T^{40/3+ \epsilon} \\
  {\rm and } \quad 
| V_{5} | & \ll x^{\frac{1}{2}+\epsilon} +x^{\frac{1}{2}+2\epsilon} \underset{2 \le T_{1} \le T}{max}  \left(
 \frac{1}{T_{1}}  \left( T_{1}^{10 \times \frac{2}{3} \times (\frac{1}{2}+\epsilon)} \right) \times \left( T_{1}^{\frac{11}{2}+\epsilon} \right)^{\frac{1}{2} + \frac{1}{2}}  \right)\\
 & \ll x^{\frac{1}{2}+\epsilon} + x^{\frac{1}{2}+2\epsilon}  T^{\frac{40}{3}+\epsilon} \ll x^{\frac{1}{2}+2\epsilon}  T^{\frac{40}{3}+\epsilon}.\\
 \end{split}
\end{equation*}
So, we get  
 \begin{equation*}
 \begin{split}
S_{5}(f, D; x ) & \ll   x^{\frac{1}{2}+ \epsilon} T^{\frac{40}{3}+\epsilon} +  \frac{x^{1+\epsilon}}{T} + x^{\frac{1}{2}+2\epsilon} T^{\frac{40}{3} + \epsilon} + \frac{x^{1+2\epsilon}}{T} . \\
       \end{split}
\end{equation*} 
Now, we choose $T = x^{\frac{3}{86}}$  to get 
\begin{equation*}
 \begin{split}
S_{5}(f, D; x ) & \ll  x^{\frac{83}{86}+ 2\epsilon}. \\
       \end{split}
\end{equation*}

{\bf Case : r =6 }
\begin{equation*}
\begin{split}
L_{6}(s) & =      \zeta(s)^{5} L(s, \chi_{D})^{5} L(s, sym^{2}f)^{8}  L(s, sym^{2}f \times \chi_{D})^{8} L(s, sym^{4}f) ^{4} L(s, sym^{4}f \times \chi_{D})^{4}  \\ & \quad \times   L(s, sym^{4}f \times sym^{2}f)  L(s, sym^{4}f \times sym^{2}f \times \chi_{D}).        \\
\end{split}
\end{equation*}
\begin{equation*}
\begin{split}
 L_{6}(s) & =  L_{6, 1}(s) \times L_{6 ,2}(s),  \quad \qquad {\rm where}  \\
&\begin{cases}
L_{6 ,1}(s)  =  \zeta(s)^{5}  L(s, sym^{2}f)^{8} L(s, sym^{4}f) ^{4} L(s, sym^{4}f \times sym^{2}f),   \\ 
 L_{ 6 ,1}(s) =  L(s, \chi_{D})^{5}   L(s, sym^{2}f \times \chi_{D})^{8} L(s, sym^{4}f \times \chi_{D})^{4} L(s, sym^{4}f \times sym^{2}f \times \chi_{D}).  
 \end{cases}
\end{split}
\end{equation*}

\noindent
Substituting these in \eqref{Ver-Est} and \eqref{Hor-Est} when $r =6$ and apply convexity bound and mean square estimate to obtain
\begin{equation*} 
 \begin{split}
 |H_{6}|& \ll \frac{x^{1+\epsilon}}{T} + \left (\frac{x^{\frac{1}{2}+ \epsilon}}{T}  T^{2 \times 5 \times ( \frac{1}{6}+\epsilon) + 2 \times \frac{8.3}{2} \times ( \frac{1}{2}+\epsilon)  + 2 \times \frac{15}{2} \times ( \frac{1}{2}+\epsilon)  }\right)
 \ll \frac{x^{1+\epsilon}}{T} + x^{\frac{1}{2}+ \epsilon} T^{181/6+ \epsilon} \\
 {\rm and } \quad  
  |V_{6}|  & \ll x^{\frac{1}{2}+ \epsilon} + x^{\frac{1}{2}+2\epsilon} \underset{2 \le T_{1} \le T}{max}  \left(
 \frac{1}{T_{1}} \times   T_{1}^{10 \times \frac{1}{6}+\epsilon}  \times \left( T_{1}^{\frac{59}{2 \times 2}+\epsilon} \right)^{\frac{1}{2} +\frac{1}{2}}  \right) \\
 &  \ll x^{\frac{1}{2}+ \epsilon} + x^{\frac{1}{2}+2\epsilon}  T^{\frac{181}{6}+\epsilon} \ll x^{\frac{1}{2}+2\epsilon}  T^{\frac{181}{6}+\epsilon}. \\
 \end{split}
\end{equation*}
Thus, we obtain 
 \begin{equation*}
 \begin{split}
S_{6}(f, D; x ) & \ll  C x P_{4}(\log x) + O\left( x^{\frac{1}{2}+ \epsilon} T^{\frac{181}{6}+\epsilon} +  \frac{x^{1+\epsilon}}{T}  \right) + O\left( x^{\frac{1}{2}+2\epsilon} T^{\frac{181}{6} + \epsilon}  \right) + O\left(\frac{x^{1+2\epsilon}}{T} \right). \\
       \end{split}
\end{equation*} 
Now, we choose $T = x^{\frac{3}{187}}$  to get 
\begin{equation*}
 \begin{split}
S_{6}(f, D; x ) & = C x P_{4}(\log x) +O\left(x^{\frac{184}{187}+ 2\epsilon} \right). \\
       \end{split}
\end{equation*}

{\bf Case : r = 7 }
\begin{equation*}
\begin{split}
L_{7}(s) & =   L(s, f)^{13} L(s, f\times \chi_{D})^{13}  L(s, sym^{3}f)^{8}  L(s, sym^{3}f \times \chi_{D})^{8} L(s, sym^{4}f \times f)^{5} \\ & \quad  \times  L(s, sym^{4}f \times f \times \chi_{D})^{5}  L(s, sym^{4}f \times sym^{3}f)  L(s, sym^{4}f \times sym^{3}f \times \chi_{D}). \\
L_{7}(s) & =  L_{7, 1}(s) \times L_{7 ,2}(s), \quad \qquad {\rm where}  \\
\end{split}
\end{equation*}
\begin{equation*}
\begin{split}
\begin{cases}
L_{7 ,1}(s)  =  L(s, f)^{13} L(s, sym^{3}f)^{8} L(s, sym^{4}f \times f)^{5} L(s, sym^{4}f \times sym^{3}f),  \\ 
 L_{7 ,1}(s) =    L(s, f\times \chi_{D})^{13}    L(s, sym^{3}f \times \chi_{D})^{8}   L(s, sym^{4}f \times f \times \chi_{D})^{5}  \\ \quad \qquad \quad  \quad \times  L(s, sym^{4}f \times sym^{3}f \times \chi_{D}). \\
   \end{cases}
\end{split}
\end{equation*}
Substituting these in \eqref{Ver-Est} and \eqref{Hor-Est} when $r =7$ and apply the convexity bound and mean square estimate the $L$-function to get
\begin{equation*} 
 \begin{split}
|H_{7} |  & \ll \frac{x^{1+\epsilon}}{T} + \left (\frac{x^{\frac{1}{2}+ \epsilon}}{T}  T^{  26 \times ({\frac{1}{3}+\epsilon} )+ 32 \times ( \frac{1}{2}+\epsilon) + 50 \times ( \frac{1}{2}+\epsilon) + 20 \times ( \frac{1}{2}+\epsilon) }\right) \\
&   \ll \frac{x^{1+\epsilon}}{T} + x^{\frac{1}{2}+ \epsilon} T^{176/3+ \epsilon} \\
{\rm and } \qquad  |V_{7}|  & \ll  x^{\frac{1}{2}+\epsilon}  +x^{\frac{1}{2}+2\epsilon} \underset{2 \le T_{1} \le T}{max}  \left(
 \frac{1}{T_{1}}  \left( T_{1}^{26  \times (\frac{1}{3}+\epsilon)}  \right) \times \left( T_{1}^{\frac{102}{2}+\epsilon} \right)^{\frac{1}{2} + \frac{1}{2}}   \right) \\
 & \ll x^{\frac{1}{2}+\epsilon} + x^{\frac{1}{2}+2\epsilon}  T^{\frac{176}{3}+\epsilon} \ll x^{\frac{1}{2}+2\epsilon}  T^{\frac{176}{3}+\epsilon}.\\
 \end{split}
\end{equation*}
Thus, we have 
 \begin{equation}
 \begin{split}
S_{7}(f, D; x ) & \ll  x^{\frac{1}{2}+ \epsilon} T^{\frac{176}{3}+\epsilon} +  \frac{x^{1+\epsilon}}{T} + x^{\frac{1}{2}+2\epsilon} T^{\frac{176}{3} + \epsilon} + \frac{x^{1+2\epsilon}}{T} . \\
       \end{split}
\end{equation} 
Now, we choose $T = x^{\frac{3}{358}}$  to obtain 
\begin{equation*}
 \begin{split}
S_{7}(f, D; x ) & \ll  x^{\frac{355}{358}+ 2\epsilon}. \\
       \end{split}
\end{equation*}

{\bf Case : r =  8}
when 
 \begin{equation*}
\begin{split}
L_{8}(s) & =       \zeta(s)^{13} L(s, \chi_{D})^{13} L(s, sym^{2}f)^{21}  L(s, sym^{2}f \times \chi_{D})^{21} L(s, sym^{4}f) ^{13} \\  & \quad \times L(s, sym^{4}f \times \chi_{D})^{13}    L(s, sym^{4}f \times sym^{2}f)^{6}  L(s, sym^{4}f \times sym^{2}f \times \chi_{D})^{6}  \\ & \quad  \times L(s, sym^{4}f \times sym^{4}f)  L(s, sym^{4}f \times sym^{4}f \times \chi_{D}).       \\
L_{8}(s) & =  L_{8, 1}(s) \times L_{8, 2}(s), \quad \qquad {\rm where}  \\
\end{split}
\end{equation*}
\begin{equation*}
\begin{split}
\begin{cases}
L_{8, 1}(s)  =  \zeta(s)^{13} L(s, sym^{2}f)^{21}  L(s, sym^{4}f) ^{13} L(s, sym^{4}f \times sym^{2}f)^{6}  L(s, sym^{4}f \times sym^{4}f), \\ 
 L_{8, 1}(s) =   L(s, \chi_{D})^{13}  L(s, sym^{2}f \times \chi_{D})^{21}  L(s, sym^{4}f \times \chi_{D})^{13}      L(s, sym^{4}f \times sym^{2}f \times \chi_{D})^{6}  \\ 
 \quad \qquad \qquad \times  L(s, sym^{4}f \times sym^{4}f \times \chi_{D}). \\
  \end{cases}
\end{split}
\end{equation*}
Substitute these in \eqref{Ver-Est} and \eqref{Hor-Est} when $r =8$. Then we apply convexity bound and mean square estimate of the $L$-function to get  
\begin{equation*} 
 \begin{split}
|H_{8}|    & \ll \frac{x^{1+\epsilon}}{T} + \left (\frac{x^{\frac{1}{2}+ \epsilon}}{T}  T^{2 \times 13 \times  ( \frac{1}{6}+\epsilon)  + 2 \times  (63+ 65+90+25) \frac{1}{2} \times ( \frac{1}{2}+\epsilon)  }\right)   \ll \frac{x^{1+\epsilon}}{T} + x^{\frac{1}{2}+ \epsilon} T^{\frac{749}{6}+ \epsilon} \\ 
{\rm and } \quad
|V_{8}| & \ll x^{\frac{1}{2}+\epsilon} + x^{\frac{1}{2}+2\epsilon} \underset{2 \le T_{1} \le T}{max}  \left(
 \frac{1}{T_{1}}  \left( T_{1}^{26 \times \frac{1}{6}+\epsilon} \right) \times \left(T_{1}^{\frac{(63+ 65+90+25)}{2}+\epsilon} \right)^{\frac{1}{2} + \frac{1}{2}}  \right) \\
 & \ll x^{\frac{1}{2}+\epsilon} + x^{\frac{1}{2}+2\epsilon}  T^{\frac{749}{6}+\epsilon}  \ll x^{\frac{1}{2}+2\epsilon}  T^{\frac{749}{6}+\epsilon}.
 \end{split}
\end{equation*}
Thus, we have 
 \begin{equation}
 \begin{split}
S_{8}(f, D; x ) & = C P_{8}(\log x) + O \left(  x^{\frac{1}{2}+ \epsilon} T^{\frac{749}{6}+\epsilon}  +  \frac{x^{1+\epsilon}}{T} \right)+ O \left( x^{\frac{1}{2}+2\epsilon} T^{\frac{749}{6} + \epsilon} \right) + O \left( \frac{x^{1+2\epsilon}}{T} \right).  \\
       \end{split}
\end{equation} 
Now, we choose $T = x^{\frac{3}{755}}$  to get 
\begin{equation*}
 \begin{split}
S_{8}(f, D; x ) & = C P_{8}(\log x)  + O \left( x^{\frac{752}{755}+ 2\epsilon} \right). \\
       \end{split}
\end{equation*}
Since $\epsilon>0$ is arbitrary, we have the required result. This completes the proof of \thmref{Estimate-BFQ}.
 
Now, we apply  \lemref{lalit-sign change} to get our result on sign change.

\bigskip
\textbf{Proof of \thmref{sign change-BFQ}} \qquad 
Let $\epsilon >0$ be a fixed arbitrary small real number.  
 We assume that  $\lambda_{f}(n)= O(n^{\epsilon/2})$ (Deligne's bound) and $r_{Q}(n) =  O(n^{\epsilon/2})$ (Weil's bound).  Now, we apply \lemref{lalit-sign change} when $r =2,$ $\textbf{S}= \mathbb{Z}$ and $g = Q(\underline{x}),$ where $Q(\underline{x})$ is a primitive integral positive definite binary quadratic form of fixed discriminant $D<0$ with the class number $1$.   From \thmref{Estimate-BFQ}, we have 
\begin{equation*}
\begin{split}
 S_{1}(f, D; x ) & =\displaystyle{\sum_{\underline{x} \in {\mathbb Z}^{2} \atop  Q(\underline{x}) \le x} \lambda_{f}(Q(\underline{x}))} =  O_{f,D,\epsilon} (x^{\frac{7}{10}+\frac{\epsilon}{2}})  \\
{\rm and} \qquad   S_{2}(f, D; x ) 
&\displaystyle{\sum_{\underline{x} \in {\mathbb Z}^{2} \atop Q(\underline{x}) \le x} (\lambda_{f}(Q(\underline{x}))^{2}} = C x + O_{f,D,\epsilon}(x^{\frac{8}{11}+\epsilon}), \\
\end{split}
\end{equation*}
where $C >0 $ is a constant.  Now, we chose $\delta = \frac{25}{33} +\epsilon$ and apply \lemref{lalit-sign change}  to obtain at least one sign change at ${\underline{x} \in {\mathbb Z}^{2} }$ such that   $Q(\underline{x}) \in (x, x+x^{\frac{25}{33} +\epsilon}].$ Thus, we have infinitely many sign changes. Moreover, we obtain at least $x^{\frac{8}{33}-\epsilon}$ many sign changes in the interval $(x, 2x],$ for sufficiently large $x$. This completes the proof.

\bigskip 

\textbf{Acknowledgement :} This work is a part of author's thesis \cite[Chapter 4]{thesisLalit}. We would like to thank Prof. B. Ramakrishnan for encouraging to work on this problem and also for his guidance and suggestions during the preparation of the manuscript. We thank HRI, Prayagraj for providing financial support through an Infosys grant.

\end{document}